\pdfoutput=1 
\documentclass[11pt]{amsart}

\usepackage{epsfig,overpic,comment}
\usepackage[usenames,dvipsnames,svgnames,table]{xcolor}
\usepackage[hyphens]{url}
\usepackage[pagebackref,linktocpage=true,colorlinks=true,linkcolor=Blue,citecolor=BrickRed,urlcolor=RoyalBlue]{hyperref}
\usepackage[msc-links,abbrev]{amsrefs}
\usepackage{amsmath,amsthm,amssymb,marginnote}
\usepackage{enumerate}
\usepackage[T1]{fontenc}

\newtheorem{theorem}{Theorem}[section]
\newtheorem{corollary}[theorem]{Corollary}

\theoremstyle{definition}
\newtheorem{note}[theorem]{Note}

\newcommand{\be}{\begin{equation}}
\newcommand{\ee}{\end{equation}}
\newcommand{\ol}{\overline}

\newcommand{\R}{\mathbf{R}}

\newcommand{\C}{\mathcal{C}}
\newcommand{\G}{\mathcal{G}}

\renewcommand{\epsilon}{\varepsilon}

\makeatletter
\DeclareFontFamily{U}{tipa}{}
\DeclareFontShape{U}{tipa}{m}{n}{<->tipa10}{}
\newcommand{\arc@char}{{\usefont{U}{tipa}{m}{n}\symbol{62}}}%

\newcommand{\arc}[1]{\mathpalette\arc@arc{#1}}

\newcommand{\arc@arc}[2]{%
  \sbox0{$\m@th#1#2$}%
  \vbox{
    \hbox{\resizebox{\wd0}{\height}{\arc@char}}
    \nointerlineskip
    \box0
  }%
}
\makeatother

\makeatletter
\def\@tocline#1#2#3#4#5#6#7{\relax
  \ifnum #1>\c@tocdepth 
  \else
    \par \addpenalty\@secpenalty\addvspace{#2}%
    \begingroup \hyphenpenalty\@M
    \@ifempty{#4}{%
      \@tempdima\csname r@tocindent\number#1\endcsname\relax
    }{%
      \@tempdima#4\relax
    }%
    \parindent\z@ \leftskip#3\relax \advance\leftskip\@tempdima\relax
    \rightskip\@pnumwidth plus4em \parfillskip-\@pnumwidth
    #5\leavevmode\hskip-\@tempdima
      \ifcase #1
       \or\or \hskip 1.3em \or \hskip 2em \else \hskip 5em \fi%
      #6\nobreak\relax
    \hfill\hbox to\@pnumwidth{\@tocpagenum{#7}}\par
    \nobreak
    \endgroup
  \fi}
\makeatother

\newcommand{\nocontentsline}[3]{}
\newcommand{\tocless}[2]{\bgroup\let\addcontentsline=\nocontentsline#1{#2}\egroup}

\begin{document}
\setlength{\baselineskip}{1.2\baselineskip}

\title[Nonpositively curved manifolds with convex boundary] 
{Rigidity of nonpositively curved manifolds\\with convex boundary}

\author{Mohammad Ghomi}
\address{School of Mathematics, Georgia Institute of Technology,
Atlanta, GA 30332}
\email{ghomi@math.gatech.edu}
\urladdr{www.math.gatech.edu/~ghomi}

\author{Joel Spruck}
\address{Department of Mathematics, Johns Hopkins University,
 Baltimore, MD 21218}
\email{js@math.jhu.edu}
\urladdr{www.math.jhu.edu/~js}

\begin{abstract}
We show that a compact Riemannian $3$-manifold $M$ with  strictly convex simply connected boundary   and sectional curvature $K\leq a\leq 0$ is isometric to a convex domain in a complete simply connected space of constant curvature $a$, provided that $K\equiv a$ on planes tangent to the boundary of $M$. This yields a characterization of strictly convex surfaces with minimal total curvature in Cartan-Hadamard $3$-manifolds, and
extends some rigidity results of Greene-Wu, Gromov, and Schroeder-Strake. Our proof is based on a recent comparison formula for total curvature of Riemannian hypersurfaces, which also yields some dual  results for $K\geq a\geq 0$.
\end{abstract}

\date{\today \,(Last Typeset)}
\thanks{The research of M.G. was supported by NSF grant DMS-2202337 and a Simons Fellowship. The research of J.S. was supported by a Simons Collaboration Grant}

\maketitle


\section{Introduction}

A \emph{Cartan-Hadamard}  manifold $\mathcal{H}$ is a complete simply connected Riemannian $n$-space with sectional curvature $K\leq 0$. Greene and Wu \cite{greene-wu1982,greene-wu1982b} and Gromov \cite[Sec. 5]{ballmann-gromov-schroeder} showed  that, when $n\geq 3$, these spaces exhibit remarkable  rigidity properties, analogous to those  observed earlier by Mok, Siu, and Yau \cite{siu-yau1977,mok-siu-yau1981} in K\"{a}hler geometry. In particular, a fundamental result is that if $K$ vanishes outside a compact set $C\subset \mathcal{H}$, then $\mathcal{H}$ is isometric to Euclidean space $\R^n$. More generally, if $K\leq a\leq 0$ on $\mathcal{H}$, and $K\equiv a$ on $\mathcal{H}\setminus C$, then $K\equiv a$ on $\mathcal{H}$ \cite[p. 734]{greene-wu1982} \cite{seshadri2009}. We extend this result when $n=3$:

\begin{theorem}\label{thm:main}
Let $M$ be a compact Riemannian $3$-manifold with  nonempty $\C^2$ boundary $\partial M$ and sectional curvature $K\leq a\leq 0$. Suppose that $\partial M$ is strictly convex, each component of $\partial M$ is simply connected,  and $K\equiv a$ on planes tangent to $\partial M$. Then $M$ is isometric to a convex domain in a Cartan-Hadamard manifold of constant curvature $a$. In particular, $M$ is diffeomorphic to a ball.
\end{theorem}

\emph{Strictly convex} here means that the second fundamental form of $\partial M$ is positive definite with respect to the outward normal. For $n=3$, this theorem immediately implies the rigidity results mentioned above, by letting $M$ be a geodesic ball in $\mathcal{H}$ containing $C$. 
Schroeder and Strake \cite{schroeder-strake1989a} had established this result for $a=0$ (and only for $n=3$) refining earlier work of Schroeder and Ziller \cite{schroeder-ziller1990}. The simply connected assumption on components of $\partial M$ is necessary, as can be seen by considering a tubular neighborhood of a closed  geodesic in a hyperbolic manifold.  

As an application of Theorem \ref{thm:main} we obtain the following characterization for strictly convex surfaces with minimal total curvature. We say that an oriented closed (compact, connected, without boundary) hypersurface $\Gamma\subset\mathcal{H}$ is \emph{strictly convex} if its second fundamental form $\mathrm{I\!I}$ is positive definite. Then $\Gamma$ is embedded, bounds a convex domain, and is simply connected \cite{alexander1977}. The \emph{total Gauss-Kronecker curvature} of $\Gamma$ is given by $\mathcal{G}(\Gamma):=\int_\Gamma \det(\mathrm{I\!I})$, and $|\Gamma|$ denotes the area of $\Gamma$.

\begin{corollary}
Let $\mathcal{H}$ be a $3$-dimensional Cartan-Hadamard manifold with curvature $K\leq a\leq 0$, and $\Gamma\subset \mathcal{H}$ be a $\C^2$ closed strictly convex surface. Then
\be\label{eq:G1}
\mathcal{G}(\Gamma)\geq 4\pi -a|\Gamma|,
\ee
with equality only if $K\equiv a$ on the convex domain bounded by $\Gamma$.
\end{corollary}
\begin{proof}
By Gauss' equation $\det(\mathrm{I\!I}_p)=K_\Gamma(p)-K(T_p\Gamma)$ for all $p\in\Gamma$, where $K_\Gamma$ is the intrinsic curvature of $\Gamma$, and $T_p\Gamma$ is the tangent plane of $\Gamma$ at $p$.
Since $\Gamma$ is simply connected, $\int_\Gamma K_\Gamma=4\pi$ by Gauss-Bonnet theorem. Thus
\be\label{eq:G2}
\mathcal{G}(\Gamma)= 4\pi-\int_{p\in\Gamma} K(T_p\Gamma)\geq 4\pi -a |\Gamma|.
\ee
 If equality holds in \eqref{eq:G1}, then it also holds in \eqref{eq:G2}, which forces $K\equiv a$ on tangent planes of $\Gamma$. Theorem \ref{thm:main}, applied to the convex domain bounded by $\Gamma$, completes the proof.
\end{proof}

For $a=0$, the last result  is stated in \cite[p. 66]{ballmann-gromov-schroeder} and follows from \cite[Thm. 2]{schroeder-strake1989a}.
Gromov's approach  to the rigidity theorems mentioned above \cite[Sec. 5]{ballmann-gromov-schroeder}, which are further developed in \cite{schroeder-strake1989a,schroeder-ziller1990}, was based on extension of isometric embeddings in locally symmetric spaces. In most of these results the rank of the space is required to be bigger than $1$, which precludes negative  upper bounds for curvature. The arguments of Greene and Wu \cite{greene-wu1982} on the other hand involve volume comparison theory, which applies readily to various curvature bounds \cite[p. 734]{greene-wu1982}; see Seshadri \cite{seshadri2009}. Here we develop a different approach
via recent work on total curvature of Riemannian hypersurfaces \cite{ghomi-spruck2022,ghomi-spruck-TMCRH}, which 
also yields some results for the dual case $K\geq a\geq 0$; see Note \ref{note:nonnegative}. 
 Generalizing \eqref{eq:G1} to dimensions $n>3$ is an important open problem with applications to the isoperimetric inequality; see \cite{ghomi-spruck2022} for more references and background in this area.

\section{Proof of Theorem \ref{thm:main}}
The proof consists of three parts. First we use the comparison formula developed  in \cite{ghomi-spruck2022,ghomi-spruck-TMCRH} to show that $K\equiv a$ on a neighborhood of $\partial M$ (which we do not a priori assume to be connected). Then it follows that $M$ is isometric to a convex domain in a Cartan-Hadamard manifold $\ol M$, which has constant curvature $a$ outside $M$ (in particular $\partial M$ is connected). Finally enclosing $M$ in a  geodesic ball $B\subset\ol M$ and shrinking the radius of $B$ completes the proof via the first part of the argument. The first part also involves the Gauss-Bonnet theorem, which is why we need to assume that $n=3$. Other aspects of the proof work in all dimensions $n\geq 3$.

\subsection*{(Part I)}
Let $\Gamma$ be a component of $\partial M$, $d_\Gamma\colon M\to\R$ be the distance function of $\Gamma$, and $\Gamma_t:=d_\Gamma^{-1}(t)$ be the parallel surface of $\Gamma$ at distance $t$. Since $\Gamma$ is $\C^2$, there exists $\epsilon>0$ such that $\Gamma_t$ is $\C^2$ for $t\in [0,\epsilon]$, see \cite{foote1984}. In particular, for $t\in [0,\epsilon]$, the principal curvatures $\kappa_i^{t}$ of $\Gamma_t$ with respect to the outward normal $\nu_t$ are well-defined. By assumption, $\kappa_i^{0}>0$, which yields that $\kappa_i^{t}>0$, assuming $\epsilon$ is sufficiently small. Let $e_i^t$ be a choice of orthogonal principal directions for $\kappa_i^t$, and
$K_i^{t}$ be the sectional curvatures of $M$ for planes spanned by $\nu_{t}$  and $e_i^{t}$. Let $\Omega_\epsilon\subset M$ be the domain bounded in between $\Gamma$ and $\Gamma_\epsilon$, and recall that $\G(\Gamma_t)$ denote the total Gauss-Kronecker curvature of $\Gamma_t$, i.e., the integral of $\kappa_1^t\kappa_2^t$ over $\Gamma_t$. Since $|\nabla d_\Gamma|$ is constant on $\Gamma_t$, the comparison formula in \cite[Thm. 3.1]{ghomi-spruck-TMCRH} reduces to
$$
\G(\Gamma)-\G(\Gamma_{\epsilon})=-\int_{\Omega_\epsilon} \big(\kappa^{t}_1K^{t}_{2}+\kappa^{t}_2K^{t}_{1}\big).
$$
Let $H^t:=\kappa^t_1+\kappa^t_2$ denote the mean curvature of $\Gamma_{t}$. Since $\kappa_i^{t}\geq 0$, 
\be\label{eq:G-Gt}
\G(\Gamma)-\G(\Gamma_{\epsilon})\geq-a\int_{\Omega_\epsilon }H^t=-a\big(|\Gamma|-|\Gamma_{\epsilon}|\big).
\ee
The last equality  is due to Stokes theorem, since $H^t=\textup{div}(\nabla d_\Gamma)$ and $|\nabla d_\Gamma|=1$ (more formally, the above inequality holds on $\Omega_\epsilon\setminus\Omega_s$, for $0<s<\epsilon$, and we may take the limit as $s\to 0^+$). On the other hand, by Gauss' equation and Gauss-Bonnet theorem,
\begin{align}\label{eq:Gt}\notag
\G(\Gamma)&=\;4\pi-\int_{p\in\Gamma} K(T_p\Gamma)=4\pi-a|\Gamma|,\\
\G(\Gamma_{\epsilon})&=4\pi-\int_{p\in\Gamma_{\epsilon}}K(T_p\Gamma_\epsilon)\geq4\pi-a|\Gamma_{\epsilon}|.
\end{align}
Hence
$$
\G(\Gamma)-\G(\Gamma_{\epsilon})\leq -a\big(|\Gamma|-|\Gamma_{\epsilon}|\big).
$$
So equality holds in \eqref{eq:G-Gt} which forces $K_i^t\equiv a$ on $\Omega_\epsilon$ for $i=1$, $2$.  Furthermore, equality in \eqref{eq:G-Gt} implies equality in \eqref{eq:Gt}, which yields that $K\equiv a$ on tangent planes of $\Gamma_\epsilon$. So $K\equiv a$ on a triplet of mutually orthogonal planes at each point of $\Gamma_\epsilon$. It follows that $K\equiv a$ with respect to all planes with footprint on $\Gamma_\epsilon$, since $K\leq a$. As this argument  holds for all $\epsilon'\leq\epsilon$, we conclude that
$K\equiv a$ on $\Omega_\epsilon$.

\subsection*{(Part II)}
Let $\mathcal{H}$ be the $3$-dimensional Cartan-Hadamard manifold  of constant curvature $a$. Then $\Omega_\epsilon$ is locally isometric to $\mathcal{H}$. Furthermore, since $\Gamma$ is simply connected, so is $\Omega_\epsilon$. Thus there exists an isometric immersion $f\colon\Omega_\epsilon\to \mathcal{H}$, by a standard monodromy argument.
In particular, $f(\Gamma)$ forms a closed immersed surface in $\mathcal{H}$ with positive principal curvatures. Consequently, by a result of Alexander \cite[Thm. 1]{alexander1977}, see also \cite[Lem. 1]{schroeder-strake1989a}, $f$ 
embeds $\Gamma$ into the boundary of a convex domain $C\subset \mathcal{H}$.  Let $C'\subset \mathcal{H}$ be the closure of $\mathcal{H}\setminus C$. Using the diffeomorphism $f$ between $f(\Gamma)=\partial C'$ and $\Gamma$,  we may glue $C'$  to $M$ along $\Gamma$ to obtain a smooth manifold with one fewer boundary component. Repeating this procedure for each component $\Gamma$ of $\partial M$ yields an extension of $M$ to a complete manifold $\ol M$ of nonpositive curvature. Now pick a point $p\in\ol M$. By Cartan-Hadamard theorem, the exponential map $\exp_p\colon T_p \ol M\to \ol M$ is a covering. Let $X$ be a component of $\ol M\setminus M$. Note that $X$ is simply connected since, by Schoenflies theorem, it is homeomorphic to the complement of a ball in $\R^3$.  Let $X'$ be a component of $\exp_p^{-1}(X)$. Since $X$ is simply connected, $X'$ is homeomorphic to $X$. In particular $\partial X'$ is an embedded topological sphere. Thus $X'$ is the complement of a bounded set in $T_p \ol M$. Since any two such sets must intersect, it follows that $\ol M\setminus M$ is connected, and  $X'=\exp_p^{-1}(X)$. In particular $\ol M\setminus M=X$, which is simply connected. Consequently $\exp_p\colon X'\to X$ is one-to-one, which yields that it is one-to-one everywhere, since $\exp_p$ is a covering map. Hence $\ol M$ is simply connected, and therefore is a Cartan-Hadamard manifold. Finally, since $\ol M\setminus M$ is connected, $\partial M$ is connected. So $M$ forms a convex domain in $\ol M$ by Alexander's result \cite[Thm. 1]{alexander1977}.

\subsection*{(Part III)}
It remains to show that $K\equiv a$ on $\ol M$. By construction $K\equiv a$ on $\ol M\setminus M$. So we just need to check that $K\equiv a$ on $M$. Fix a point $o$ in $\ol M$ and let $B_r\subset\ol M$ be the geodesic ball of radius $r$ centered at $o$.  If $r$ is large enough, so that $M\subset B_r$, then $K\equiv a$ outside $B_r$. Let $r_0$ be the infimum of $r>0$ such that $K\equiv a$ on $\ol M\setminus B_r$. If $r_0=0$ we are done. Otherwise, since $a\leq 0$, $\partial B_{r_0}$ forms a strictly convex  surface by Hessian comparison (the principal curvatures of $\partial B_{r_0}$ are bigger than those of a sphere of the same radius in $\R^3$  \cite[1.7.3]{karcher1989}). Thus we may apply 
the result of Part I to $B_{r_0}$ to obtain that $K\equiv a$ on a neighborhood of $\partial B_{r_0}$ in $B_{r_0}$, which contradicts the definition of $r_0$. So we conclude that $K\equiv a$ everywhere, which completes the proof.

\section{Notes}

\begin{note}\label{note:nonnegative}
Part I of the proof of Theorem \ref{thm:main} works just as well for nonnegative curvature, i.e., suppose that $K\geq a\geq 0$ on $M$ and $K\equiv a$ on tangent planes of $\partial M$, then virtually the same argument shows that $K\equiv a$ on an open neighborhood of $\partial M$. Thus if $\partial M$ contracts to a point through strictly convex surfaces, then $K\equiv a$ on $M$ as we showed in Part III. This may be considered a dual version of Theorem \ref{thm:main}. For instance if $M$ is a geodesic ball of radius $r$ in a space with $K\leq b$, then it satisfies the contraction property provided that $r\leq \pi/(2\sqrt{b})$, by  Hessian comparison \cite[1.7.3]{karcher1989}. More generally, the required contraction may be achieved via a curvature flow when maximum value of $K$ is not too large compared to principal curvatures of $\partial M$  \cite{andrews1994}. Furthermore note that when  $a=0$, and $\partial M$ is simply connected, $M$ may be extended to a nonnegatively curved manifold $\ol M$ which is flat outside $M$, as discussed in Part II of the above proof. Then $\ol M$ is isometric to $\R^3$ by \cite[Thm. 1]{greene-wu1982}.  So $M$ will be flat, as had been observed earlier in \cite[p. 486]{schroeder-strake1989a}.  See \cite{schroeder-ziller1990,schroeder-strake1989b} for 
more rigidity results for nonnegative curvature 
\end{note}

\begin{note}
Once Part I of the  proof of Theorem \ref{thm:main} has been established, and it is known a priori that $M$ is simply connected with connected boundary $\Gamma$, one may complete the argument more directly by covering $M$ with a continuous family of strictly convex surfaces $\Gamma_t$ with $\Gamma_0=\Gamma$. For instance, we may let $\Gamma_t$ be  level sets of a strictly convex function on $M$, see \cite[Lem. 1]{borbely2002}. Alternatively, one may use harmonic mean curvature flow, i.e., set $\Gamma_t:=X(\Gamma,t)$ for $X\colon \Gamma\times[0,T)\to M$  given by
$$
\frac{\partial}{\partial t} X(p,t)=\frac{-1}{1/\kappa_1^t(p)+1/\kappa_2^t(p)}\nu_t(p),\quad\quad X(p,0)=p,
$$
where $\nu_t$ is the outward normal of $\Gamma_t$ and $\kappa_i^t$ are its principal curvatures.
Xu showed \cite{xu2010}, see also Gulliver and Xu \cite{gulliver-xu2009},   that $\Gamma_t$ converges to a point $o$ as $t\to T$, and  remains strictly convex throughout \cite[Prop. 19]{xu2010}. So $\Gamma_t$ always moves inward, foliating the region $M\setminus\{o\}$. The stated regularity requirement in \cite{xu2010} is that $\Gamma$ be $\C^\infty$, which we may assume is the case after a perturbation of $\Gamma$ \cite[Lem. 3.3]{ghomi-spruck-minkowski}, since by Part I we have $K\equiv a$ near $\Gamma$.
\end{note}

\section*{Acknowledgments}
We thank  Igor Belegradek and Joe Hoisington for useful communications.

\addtocontents{toc}{\protect\setcounter{tocdepth}{0}}

\addtocontents{toc}{\protect\setcounter{tocdepth}{1}}
\bibliography{references}

\begin{bibdiv}
\begin{biblist}

\bib{alexander1977}{article}{
      author={Alexander, S.},
       title={Locally convex hypersurfaces of negatively curved spaces},
        date={1977},
        ISSN={0002-9939},
     journal={Proc. Amer. Math. Soc.},
      volume={64},
      number={2},
       pages={321\ndash 325},
         url={https://doi-org.prx.library.gatech.edu/10.2307/2041451},
      review={\MR{448262}},
}

\bib{andrews1994}{article}{
      author={Andrews, Ben},
       title={Contraction of convex hypersurfaces in {R}iemannian spaces},
        date={1994},
        ISSN={0022-040X},
     journal={J. Differential Geom.},
      volume={39},
      number={2},
       pages={407\ndash 431},
         url={http://projecteuclid.org/euclid.jdg/1214454878},
      review={\MR{1267897}},
}

\bib{ballmann-gromov-schroeder}{book}{
      author={Ballmann, Werner},
      author={Gromov, Mikhael},
      author={Schroeder, Viktor},
       title={Manifolds of nonpositive curvature},
      series={Progress in Mathematics},
   publisher={Birkh\"{a}user Boston, Inc., Boston, MA},
        date={1985},
      volume={61},
        ISBN={0-8176-3181-X},
         url={https://doi.org/10.1007/978-1-4684-9159-3},
      review={\MR{823981}},
}

\bib{borbely2002}{article}{
      author={Borb\'{e}ly, Albert},
       title={On the total curvature of convex hypersurfaces in hyperbolic
  spaces},
        date={2002},
        ISSN={0002-9939},
     journal={Proc. Amer. Math. Soc.},
      volume={130},
      number={3},
       pages={849\ndash 854},
  url={https://doi-org.prx.library.gatech.edu/10.1090/S0002-9939-01-06101-9},
      review={\MR{1866041}},
}

\bib{foote1984}{article}{
      author={Foote, Robert~L.},
       title={Regularity of the distance function},
        date={1984},
        ISSN={0002-9939},
     journal={Proc. Amer. Math. Soc.},
      volume={92},
      number={1},
       pages={153\ndash 155},
         url={https://doi.org/10.2307/2045171},
      review={\MR{749908}},
}

\bib{ghomi-spruck-minkowski}{article}{
      author={Ghomi, Mohammad},
      author={Spruck, Joel},
       title={Minkowski inequality in cartan-hadamard manifolds},
        date={2022},
     journal={arXiv:2206.06554},
}

\bib{ghomi-spruck2022}{article}{
      author={Ghomi, Mohammad},
      author={Spruck, Joel},
       title={Total curvature and the isoperimetric inequality in
  {C}artan-{H}adamard manifolds},
        date={2022},
        ISSN={1050-6926},
     journal={J. Geom. Anal.},
      volume={32},
      number={2},
       pages={Paper No. 50, 54pp},
         url={https://doi.org/10.1007/s12220-021-00801-2},
      review={\MR{4358702}},
}

\bib{ghomi-spruck-TMCRH}{article}{
      author={Ghomi, Mohammad},
      author={Spruck, Joel},
       title={Total mean curvatures of {R}iemannian hypersurfaces},
        date={2022},
     journal={arXiv:2204.07624, to appear in Adv. Nonliner Stud.},
}

\bib{greene-wu1982}{article}{
      author={Greene, R.~E.},
      author={Wu, H.},
       title={Gap theorems for noncompact {R}iemannian manifolds},
        date={1982},
        ISSN={0012-7094},
     journal={Duke Math. J.},
      volume={49},
      number={3},
       pages={731\ndash 756},
  url={http://projecteuclid.org.prx.library.gatech.edu/euclid.dmj/1077315386},
      review={\MR{672504}},
}

\bib{greene-wu1982b}{article}{
      author={Greene, R.~E.},
      author={Wu, H.},
       title={On a new gap phenomenon in {R}iemannian geometry},
        date={1982},
        ISSN={0027-8424},
     journal={Proc. Nat. Acad. Sci. U.S.A.},
      volume={79},
      number={2},
       pages={714\ndash 715},
         url={https://doi.org/10.1073/pnas.79.2.714},
      review={\MR{648065}},
}

\bib{gulliver-xu2009}{article}{
      author={Gulliver, Robert},
      author={Xu, Guoyi},
       title={Examples of hypersurfaces flowing by curvature in a {R}iemannian
  manifold},
        date={2009},
        ISSN={1019-8385},
     journal={Comm. Anal. Geom.},
      volume={17},
      number={4},
       pages={701\ndash 719},
  url={https://doi-org.prx.library.gatech.edu/10.4310/CAG.2009.v17.n4.a6},
      review={\MR{2601350}},
}

\bib{karcher1989}{incollection}{
      author={Karcher, Hermann},
       title={Riemannian comparison constructions},
        date={1989},
   booktitle={Global differential geometry},
      series={MAA Stud. Math.},
      volume={27},
   publisher={Math. Assoc. America, Washington, DC},
       pages={170\ndash 222},
      review={\MR{1013810}},
}

\bib{mok-siu-yau1981}{article}{
      author={Mok, Ngaiming},
      author={Siu, Yum~Tong},
      author={Yau, Shing~Tung},
       title={The {P}oincar\'{e}-{L}elong equation on complete {K}\"{a}hler
  manifolds},
        date={1981},
        ISSN={0010-437X},
     journal={Compositio Math.},
      volume={44},
      number={1-3},
       pages={183\ndash 218},
         url={http://www.numdam.org/item?id=CM_1981__44_1-3_183_0},
      review={\MR{662462}},
}

\bib{schroeder-ziller1990}{article}{
      author={Schroeder, V.},
      author={Ziller, W.},
       title={Local rigidity of symmetric spaces},
        date={1990},
        ISSN={0002-9947},
     journal={Trans. Amer. Math. Soc.},
      volume={320},
      number={1},
       pages={145\ndash 160},
         url={https://doi-org.prx.library.gatech.edu/10.2307/2001755},
      review={\MR{958901}},
}

\bib{schroeder-strake1989a}{article}{
      author={Schroeder, Viktor},
      author={Strake, Martin},
       title={Local rigidity of symmetric spaces of nonpositive curvature},
        date={1989},
        ISSN={0002-9939},
     journal={Proc. Amer. Math. Soc.},
      volume={106},
      number={2},
       pages={481\ndash 487},
         url={https://doi-org.prx.library.gatech.edu/10.2307/2048830},
      review={\MR{929404}},
}

\bib{schroeder-strake1989b}{article}{
      author={Schroeder, Viktor},
      author={Strake, Martin},
       title={Rigidity of convex domains in manifolds with nonnegative {R}icci
  and sectional curvature},
        date={1989},
        ISSN={0010-2571},
     journal={Comment. Math. Helv.},
      volume={64},
      number={2},
       pages={173\ndash 186},
         url={https://doi-org.prx.library.gatech.edu/10.1007/BF02564668},
      review={\MR{997359}},
}

\bib{seshadri2009}{article}{
      author={Seshadri, Harish},
       title={An elementary approach to gap theorems},
        date={2009},
        ISSN={0253-4142},
     journal={Proc. Indian Acad. Sci. Math. Sci.},
      volume={119},
      number={2},
       pages={197\ndash 201},
  url={https://doi-org.prx.library.gatech.edu/10.1007/s12044-009-0020-5},
      review={\MR{2526423}},
}

\bib{siu-yau1977}{article}{
      author={Siu, Yum~Tong},
      author={Yau, Shing~Tung},
       title={Complete {K}\"{a}hler manifolds with nonpositive curvature of
  faster than quadratic decay},
        date={1977},
        ISSN={0003-486X},
     journal={Ann. of Math. (2)},
      volume={105},
      number={2},
       pages={225\ndash 264},
         url={https://doi-org.prx.library.gatech.edu/10.2307/1970998},
      review={\MR{437797}},
}

\bib{xu2010}{book}{
      author={Xu, Guoyi},
       title={Harmonic mean curvature flow in {R}iemannian manifolds and
  {R}icci flow on noncompact manifolds},
   publisher={ProQuest LLC, Ann Arbor, MI},
        date={2010},
        ISBN={978-1124-04603-7},
  url={http://gateway.proquest.com.prx.library.gatech.edu/openurl?url_ver=Z39.88-2004&rft_val_fmt=info:ofi/fmt:kev:mtx:dissertation&res_dat=xri:pqdiss&rft_dat=xri:pqdiss:3408443},
        note={Thesis (Ph.D.)--University of Minnesota},
      review={\MR{2941371}},
}

\end{biblist}
\end{bibdiv}

\end{document}